\documentclass[10pt]{article}
\usepackage{epsf,epsfig,amsfonts,amsgen,amsmath,amstext,amsbsy,amsopn,amsthm
}
\usepackage{ebezier,eepic}
\usepackage{color}
\newtheorem{thm}{Theorem}[section]

\newtheorem{lem}{Lemma}
\newtheorem{false statement}{False statement}

\newtheorem{fact}{Fact}

\theoremstyle{def}
\newtheorem{definition}{Definition}
\newtheorem{claim}{Claim}

\newtheorem{conj}{Conjecture}
\newcommand{\ex}{{\rm ex}}

\newcounter{mathitem}
  {\begin{list}{{$(\roman{mathitem})$}}{
   \setcounter{mathitem}{0}
   \usecounter{mathitem}
   \setlength{\topsep}{0pt plus 2pt minus 0pt}
   \setlength{\parskip}{0pt plus 2pt minus 0pt}
   \setlength{\partopsep}{0pt plus 2pt minus 0pt}
   \setlength{\parsep}{0pt plus 2pt minus 0pt}
   \setlength{\leftmargin}{35pt}
   \setlength{\itemsep}{0pt plus 2pt minus 0pt}}}
  {\end{list}}

\begin{document}

\title{\bf\Large Extensions of  the Erd\H{o}s-Gallai Theorem and Luo's Theorem with Applications}

\date{}

\author{Bo Ning\footnote{Center for Applied Mathematics, Tianjin University,
Tianjin, 300062 P.R. China. Email: bo.ning@tju.edu.cn. Supported by the NSFC grants (No.\ 11601379,
No.\ 11771141) and the Seed Foundation of Tianjin University (2018XRG-0025).}
~and  Xing Peng\footnote{Center for Applied Mathematics, Tianjin University,
Tianjin, 300062 P.R. China. Email: x2peng@tju.edu.cn. Supported by the NSFC grant
(No.\ 11601380) and the Seed Foundation of Tianjin University (2017XRX-0011)}\\[2mm]}
\maketitle

\begin{abstract}
The famous Erd\H{o}s-Gallai Theorem on the Tur\'an number of paths states that every graph
with $n$ vertices and $m$ edges contains a path with at least
$\frac{2m}{n}$ edges. In this note, we first establish a simple but novel
extension of the  Erd\H{o}s-Gallai Theorem  by proving that every graph
$G$ contains a path with at least $\frac{(s+1)N_{s+1}(G)}{N_{s}(G)}+s-1$ edges,
where $N_j(G)$ denotes the number of $j$-cliques in $G$ for $1\leq j\leq\omega(G)$.
We also construct a family of graphs which shows our extension improves the
estimate given by Erd\H{o}s-Gallai Theorem. Among applications, we show,
for example, that the main results of \cite{L17}, which are on the maximum
possible number of $s$-cliques in an $n$-vertex graph without a path with
$l$ vertices (and without cycles of length at least $c$), can be easily deduced
from this extension. Indeed, to prove these results, Luo \cite{L17}
generalized a classical theorem of Kopylov and established a tight
upper bound on the number of $s$-cliques in an $n$-vertex 2-connected
graph with circumference less than $c$. We prove a similar result for an
$n$-vertex 2-connected graph with circumference less than $c$ and large
minimum degree. We conclude this paper with an application of our results
to a problem from spectral extremal graph theory on consecutive lengths of
cycles in graphs.

\medskip
\noindent {\bf Keywords:} Clique; Cycle; Generalized Tur\'an number;
Erd\H{o}s-Gallai Theorem; Kopylov's theorem
\smallskip
\end{abstract}

\section{The Erd\H{o}s-Gallai Theorem and an extension}
Let $\mathcal{H}$ be a family of graphs.
The Tur\'an number  $\ex(n,\mathcal{H})$ is the largest possible number of
 edges in an $n$-vertex  graph $G$ which
contains  no member of $\cal H$  as a subgraph. If $\mathcal{H}=\{H\}$,
then we write $\ex(n,H)$ for  $\ex(n,\mathcal{H})$. We use $P_l$ to denote
a path with $l$ vertices. In this case, we say $P_l$ is of length $l-1$.

Erd\H{o}s and Gallai  \cite{EG59}
proved the following celebrated theorems
on Tur\'an numbers of cycles and paths.

\begin{thm}[Erd\H{o}s and Gallai \cite{EG59}]\label{Thm_EGcycle}
$\ex(n,\mathcal{C}_{\geq l}) \leq \frac{(l-1)(n-1)}{2}$, where $l\geq 3$ and  $\mathcal{C}_{\geq l}$
is the set of all cycles of length at least $l$.
\end{thm}

\begin{thm}[Erd\H{o}s and Gallai \cite{EG59}]\label{Thm_EGpath}
$\ex(n,P_l) \leq \frac{(l-2)n}{2}$, where $l \geq 2$.
\end{thm}
For the tightness of Theorem \ref{Thm_EGcycle}, one can check
the graph consisting of $\tfrac{n-1}{l-2}$ cliques of size $l-1$
with a common vertex, where $n-1$ is divisible by $l-2$.  The tightness
of Theorem \ref{Thm_EGpath} is shown by the graph
with  $\tfrac{n}{l-1}$ disjoint $K_{l-1}$, where
$n$ is divisible by $l-1$. For more improvements and extensions of Erd\H{o}s-Gallai's theorems,
see \cite{B71,W72,L75,FS75,W76,F90,BF91,CV91}. We refer the reader to an
excellent survey on related topics by F\"{u}redi and Simonovits \cite{FS13}.

For a graph $G$, let $\omega(G)$ be the {\it clique number} of $G$, i.e.,
the size of a largest clique in $G$. For $1 \leq j \leq \omega(G)$, we use $N_j(G)$ to denote the number
of copies of $K_j$ in $G$.  Recall  Theorem \ref{Thm_EGpath} can be rephrased as each graph contains a path
of length at least $\tfrac{2N_2}{N_1}$.
The main purpose of this note is to prove the following extension of Theorem \ref{Thm_EGpath}
and present several applications of this result. Since the proof of the following theorem is very short, we prove it right after we state it.
\begin{thm}\label{Thm_cliquepath}
Let $G$ be a graph. For each positive integer $s$ with $1\leq s\leq\omega(G)$, there is
a path of length at least $\frac{(s+1)N_{s+1}(G)}{N_{s}(G)}+s-1$ in $G$.
\end{thm}
\begin{proof}
We prove the theorem by induction on $s$. The case of $s=1$ is Theorem
\ref{Thm_EGpath}. Suppose it is true for $s=k-1$, where $s\leq\omega(G)-1$.
For each vertex $x\in V(G)$, let $G_x$ be the subgraph induced
by $N_G(x)$, and $l_x$ be the length of a longest path in $G_x$.
By induction hypothesis, for each vertex $x\in V(G)$
with $N_{k-1}(G_x)\neq 0$,
$ l_x \geq \frac{kN_{k}(G_x)}{N_{k-1}(G_x)}+k-2$.  Equivalently, $(l_x-k+2)N_{k-1}(G_x) \geq kN_k(G_x)$.
Let $l_{\max}=\max \{l_x: x\in V(G)\}$.  Then
\begin{equation} \label{eachx}
(l_{\max}-k+2)N_{k-1}(G_x) \geq kN_k(G_x)
\end{equation}
 holds for each $x$.
For $i \in \{k-1,k\}$,  let $V_{i}:=\{x\in V(G): N_{i}(G_x)\neq 0\}$.  Summing
inequality \eqref{eachx} over all $x \in V_{k-1}$,  we get
$$
(l_{\max}-k+2) \sum_{x \in V_{k-1}} N_{k-1}(G_x) \geq k \sum_{x \in V_{k-1}} N_k(G_x).
$$
Note that  $\sum_{x\in V_{k-1}}N_{k-1}(G_x)=kN_{k}(G)$ and
$\sum_{x\in V_{k}}N_{k}(G_x)=(k+1)N_{k+1}(G)$. It is easy to observe $V_{k}\subseteq V_{k-1}$.
By definition,  $N_{k}(G_y)=0$ for each $y\in V_{k-1}\backslash V_{k}$.
Thus $\sum_{x\in V_{k}}N_{k}(G_x)=\sum_{x\in  V_{k-1}}N_k(G_x)$.  We get
 $$
kN_k(G)(l_{\max}-k+2)  \geq k \sum_{x \in V_{k-1}} N_k(G_x)=k  \sum_{x \in V_{k}} N_k(G_x)=k(k+1) N_{k+1}(G).
$$
So $l_{\max} \geq \tfrac{(k+1)N_{k+1}(G)}{N_k(G)}+k-2$.
This implies that there exists a vertex $v$  such that $G_v$ contains a path $P_v$ of
length at least $\tfrac{(k+1)N_{k+1}(G)}{N_k(G)}+k-2$. Therefore,  there is a path of
length at least $\frac{(k+1)N_{k+1}(G)}{N_{k}(G)}+k-1$ in
$G$. The proof is complete.
\end{proof}
The following family of graphs shows our extension improves
the estimate given by Theorem \ref{Thm_EGpath}. Let $G$ be an $n$-vertex
graph which consists of a $K_{n-2}$ and  two pendant edges sharing an endpoint from the $K_{n-2}$.
Theorem \ref{Thm_EGpath} implies that $G$ contains a path of length at least
$\frac{2N_2(G_1)}{N_1(G_1)}=n-5+\frac{10}{n}$; while Theorem \ref{Thm_cliquepath}
tells us that $G$ contains a path of length at least
$\frac{(n-2)N_{n-2}(G)}{N_{n-3}(G)}+n-4=n-3$,
where we choose $s=n-3$.

For two graphs $G$ and $H$, we write $G\vee H$ for their join  which satisfies
$V(G\vee  H)=V(G)\cup V(H)$ and $E(G\vee H)=E(G) \cup E(H) \cup \{xy:x\in V(G), y\in V(H)\}$.
The proof of Theorem \ref{Thm_cliquepath} implicitly implies the following result.
\begin{thm}\label{Thm_Generkipas}
Let $\omega(G)\geq k\geq 2$ be an integer. If  $G$ is a graph with $N_k(G)\neq 0$,
then $G$ contains  a subgraph $P_l\vee K_1$, where
$l\geq\tfrac{(k+1)N_{k+1}(G)}{N_{k}(G)}+k-1$. In particular, $G$ contains
cycles of lengths from $3$ to
$\left\lceil\tfrac{(k+1)N_{k+1}(G)}{N_{k}(G)}\right\rceil+k$.
\end{thm}


\section{Short proofs of two theorems of Luo}
Before we present applications of Theorems  \ref{Thm_cliquepath} and \ref{Thm_Generkipas}  to the generalized Tur\'an number, we recall a few definitions.
Let $T$ be a graph and $\mathcal{H}$ be a family of graphs. The generalized Tur\'an
number $\ex(n,T,\mathcal{H})$ is  the maximum possible number of copies of $T$
in an $n$-vertex graph which is $H$-free for each $H\in \mathcal{H}$.
When $\mathcal{H}=\{H\}$, we
write $\ex(n,T,H)$ instead of $\ex(n,T,\{H\})$. If $T=K_2$,  then
$\ex(n,K_2,H)=\ex(n,H)$ is the classical Tur\'an number of $H$.

The generalized Tur\'an
number  has received a lot of attention recently. There are several
notable and nice papers concerning the generalized Tur\'{a}n
number $\ex(n,T,H)$ (see \cite{E62-2,BG08,G12,FS13,AS16,L17,EGMS-Arxiv}).
Erd\H{o}s \cite{E62-2} first determined $\ex(n,K_t,K_r)$ for all $t<r$.
Bollob\'{a}s and Gy\H{o}ri \cite{BG08} determined the order of magnitude
of $\ex(n,C_3,C_5)$.  Their estimate was improved by Alon and Shikhelman \cite{AS16}
and recently by Ergemlidze et al. \cite{EGMS-Arxiv}. Alon and Shikhelman
obtained a number of results on $\ex(n,T,H)$ for different $T$ and $H$
and posed several open problems in \cite{AS16}.

Luo \cite{L17} recently proved  upper bounds for
$\ex(n,K_s,\mathcal{C}_{\geq l})$ and $\ex(n,K_s,P_l)$  which are
generalizations of Theorem \ref{Thm_EGcycle} and Theorem \ref{Thm_EGpath}.

\begin{thm}[Luo \cite{L17}] \label{luo1}
$\ex(n,K_s,\mathcal{C}_{\geq l})  \leq \tfrac{n-1}{l-2}\binom{l-1}{s}$, where $l\geq 3$ and $s\geq 2$.
\end{thm}

\begin{thm}[Luo \cite{L17}]\label{luo2}
$\ex(n,K_s,P_l) \leq \frac{n}{l-1}\binom{l-1}{s}$, where $l\geq 2$ and $s\geq 2$.
\end{thm}

Luo's result turned out to be useful for investigating Tur\'an-type problems
in hypergraphs. For example, Gy\H{o}ri, Methuku, Salia, Tompkins, and
Vizer \cite{GMSTV18} applied Theorem \ref{luo1} to study the maximum number
of hyperedges in a connected $r$-uniform  $n$-vertex hypergraph without a
Berge path of length $k$.

We next give  very shorts proofs of Theorems \ref{luo1}
and \ref{luo2} by applying Theorems \ref{Thm_Generkipas} and \ref{Thm_cliquepath}
 respectively.
\smallskip
\noindent
{\bf A short proof of Theorem \ref{luo1}.}
Let $c$ be the length of a longest cycle in $G$.
By Theorem \ref{Thm_Generkipas} and the condition in Theorem \ref{luo1}, we have
$ \frac{kN_{k}(G)}{N_{k-1}(G)}+k-1 \leq c \leq l-1  $. This implies  $N_k(G)\leq  \frac{l-k}{k}N_{k-1}(G)$ holds for $3\leq k\leq s$.
We apply the inequality recursively and get
$$N_s(G)\leq \frac{(l-s)(l-s+1)\cdots(l-3)}{s(s-1)\cdots3}N_2(G).$$
By Theorem \ref{Thm_EGcycle}, we have $N_2(G)\leq \frac{(n-1)(l-1)}{2}$, and thus
$N_s(G)\leq \frac{n-1}{l-2}\binom{l-1}{s}$.
This completes the proof. {\hfill$\Box$}
\smallskip
\noindent
{\bf A short proof of Theorem \ref{luo2}.}
Since $G$ is $P_l$-free,  the length of a longest path  $P$  in $G$ is at most $ l-2$.
By Theorem \ref{Thm_cliquepath}, we have $l-2\geq \frac{kN_k(G)}{N_{k-1}(G)}+k-2$ whenever $2\leq k\leq s$.
It follows $N_k(G)\leq  \tfrac{l-k}{k}N_{k-1}(G)$ for $2 \leq k \leq s.$   Recursively applying this inequality, we get
$$N_s(G) \leq \frac{(l-s)(l-s+1)\cdots(l-3)}{s(s-1)\cdots 3}N_2(G).$$
Theorem \ref{Thm_EGpath} gives $N_2(G) \leq \frac{(l-2)n}{2}$ and so $N_s(G) \leq \frac{n}{l-1} \binom{l-1}{s}$.
This completes the proof. {\hfill$\Box$}

\section{An extension of Luo's theorem}

In order to prove Theorems \ref{luo1} and \ref{luo2},
Luo \cite{L17} extended some classical theorems due to Kopylov \cite{K77}.
Let $H(n,k,c)$ be a graph obtained from $K_{c-k}$ by connecting each
vertex of a set of $n-(c-k)$ isolated vertices to the same $k$ vertices
choosing from $K_{c-k}$. Let $f_s(n,k,c)$ be the number of $K_s$ in
$H(n,k,c)$. Namely, $f_s(n,k,c)=\binom{c-k}{s}+\binom{k}{s-1}\left(n-(c-k) \right)$.
When $s=2$, it equals the number of edges in $H(n,k,c)$. The circumference of a graph $G$ is the length of a longest cycle in $G$.
Improving Theorem \ref{Thm_EGcycle}, Kopylov \cite{K77} proved
the following.

\begin{thm}[Kopylov \cite{K77}]\label{Thm_Kopy}
Let $n\geq c\geq 5$ and $G$ be a 2-connected graph on $n$ vertices
with circumference less than $c$. Then
$N_2(G)\leq \max\{f_2(n,2,c),f_2(n,\lfloor\frac{c-1}{2}\rfloor,c)\}$.
\end{thm}

Kopylov's theorem was reproved by Fan, Lv and Wang in \cite{FLW04} who indeed proved
a slightly stronger result with the aid of another result
of Woodall \cite{W76}. In the same paper \cite{W76}, Woodall posed a  conjecture
which is  a generalization of a previous result on nonhamiltonian graphs
due to Erd\H{o}s \cite{E62}.

\begin{conj}[Woodall \cite{W76}]\footnote{It should be mentioned that, in the last part of the paper of Kopylov,
he wrote a sentence as follows: ``we remark that a proof of Woodall's conjecture
can be obtained by a minor modification of the solution to
Problem D." (quoted from \cite{K77}).}
Let $n\geq c\geq 5$. If $G$ is a 2-connected graph on $n$ vertices
with circumference less than $c$ and minimum degree $\delta(G)\geq k$, then
$N_2(G)\leq \max\{f_2(n,k,c),f_2(n,\lfloor\frac{c-1}{2}\rfloor,c)\}$.
\end{conj}

One can easily find that Kopylov's theorem confirmed Woodall's conjecture for $k=2$.

Generalizing  Kopylov's result, Luo \cite{L17}
proved the following theorem.
\begin{thm}[Luo \cite{L17}] \label{Thm_Lclipat}
Let $n\geq c\geq 5$ and $s\geq 2$. If  $G$ is a 2-connected graph on $n$ vertices
with circumference less than $c$, then
$N_s(G)\leq \max\{f_s(n,2,c), f_s(n,\lfloor\frac{c-1}{2}\rfloor,c)\}$.
\end{thm}

We present an extension of Theorem \ref{Thm_Lclipat},
which is in the spirit of Kopylov's remark (see the footnote).
\begin{thm} \label{Thm_KWLcycle}
Let $n\geq c\geq 5$ and $s\geq 2$. If  $G$ is a 2-connected graph on $n$ vertices
with circumference less than $c$ and minimum degree $\delta(G)\geq k\geq 2$, then
$N_s(G)\leq \max\{f_s(n,k,c),f_s(n,\lfloor\frac{c-1}{2}\rfloor,c)\}$.
\end{thm}

To prove Theorem \ref{Thm_KWLcycle}, we need the following lemma,  whose proof is omitted in \cite{K77}.
We would like to mention that this
generalizes Bondy's lemma on longest cycles,
whose proof is implicit in the proof of Lemma
1 in \cite{B71}.
\begin{lem}[Kopylov \cite{K77}]\label{Lem_Kopylov}
Let $G$ be a 2-connected $n$-vertex graph with a path $P$ of $m$
edges with endpoints $x$ and $y$. For $v\in V(G)$, let $d_P(v)=|N(v)\cap V (P)|$. Then $G$ contains
a cycle of length at least $\min\{m+1,d_P(x)+d_P(y)\}$.
\end{lem}
We also need a definition from
Kopylov \cite{K77}.
\begin{definition}
[$\alpha$-disintegration of a graph, Kopylov \cite{K77}]
Let $G$ be a graph and $\alpha$ be a natural number. Delete all
vertices of degree at most $\alpha$ from $G$; for the resulting graph $G'$,
we again delete all vertices of degree at most  $ \alpha$ from $G'$. We keep running this process  until we finally get a graph, denoted by $H(G;\alpha)$, such
that all vertices are of degree  larger than $\alpha$.
\end{definition}

Our proof is  very similar to Kopylov's proof \cite{K77}
of Theorem \ref{Thm_Kopy} and  the proof of
Theorem \ref{Thm_Lclipat} in \cite{L17}.
We only give the sketch and omit the details.
We split the proof into five steps.

\smallskip
\noindent
{\bf A sketch of the proof of Theorem \ref{Thm_KWLcycle}.}
Let $G$ be a counterexample such that $G$ is edge maximal, i.e., adding
each nonedge creates a cycle of length at least $c$. Thus each pair of
nonadjacent vertices  is  connected by a path of length
at least $c-1$. Let $t=\lfloor \tfrac{c-1}{2} \rfloor$ and $H=H(G;t)$.
\begin{claim}[\cite{L17}]
$H$ is not empty.
\end{claim}
\begin{proof}
Suppose not. For the first $n-t$ vertices in the process of getting
 $H(G;t)$, each of them has degree at most $t$ and then it is
contained in at most  $(n-t) \binom{t}{s-1}$ copies of $K_s$.
The number of copies of $K_s$ in the subgraph induced by the last $t$ vertices
is bounded from above by $\binom{t}{s}$. Thus we have the following upper bound on $N_s(G)$:
\[
N_s(G) \leq  (n-t) \binom{t}{s-1}+\binom{t}{s} \leq f_s(n,t,c),
\]
which is a contradiction.
\end{proof}

\begin{claim}[\cite{K77}]
$H$ is a clique.
\end{claim}
The main differences come from  Claims 3 and 4, whose proofs need the minimum
degree condition and a new function.
\begin{claim}\label{claim-inequality}
Let $r=|V(H)|$. Then $k\leq c-r \leq t$.
\end{claim}
\begin{proof}
As $H=H(G;t)$ is a clique, $r \geq t+2$. We first claim $r\leq c-k$,
where $\delta(G)\geq k$.  Suppose  $r\geq c-k+1$.
If $x\in V(G)\setminus V(H)$,
then $x$ is not adjacent to at least one vertex in $H$. Otherwise,
$x\in H$. We pick $x\in V(G)\setminus V(H)$ and $y\in V(H)$
satisfying the following two conditions: (a) $x$ and $y$ are
not adjacent; and  (b) a  longest path in $G$ from $x$ to $y$ contains
the largest number of edges among such nonadjacent pairs. Let
$P$ be a longest path in $G$ from $x$ to $y$. Clearly,
$|V(P)|\geq c$ as $G$ is edge maximal. We next show $N_G(x)\subseteq V(P)$.
Suppose not.  Let  $z \in N_G(x)$ and $z \notin V(P)$. If $z$ and $y$ are not
adjacent, then there is a longer path from $z$ to $y$,
a contradiction to the selection of $x$ and $y$.  If $z$ and $y$ are
adjacent, then there is a cycle of length at least $c+1$,
a contradiction to the assumption of $G$. Similarly,
we can show $N_H(y)\subseteq V(P)$. Therefore, by Lemma \ref{Lem_Kopylov},
there is a cycle with length at least
$\min\{c,d_{P}(x)+d_P(y)\}\geq \min \{c,k+c-k\}=c$,
a contradiction. Thus $r \leq c-k$.
Recall $t+2\leq r\leq c-k$. We get $k\leq c-r\leq c-t-2\leq t$.
This proves Claim \ref{claim-inequality}.
\end{proof}
\begin{claim}\label{claim-convex}
Let $H'=H(G;c-r)$. Then $H \neq H'$.
\end{claim}
\begin{proof}
If $H=H'$, then  we have
\[
N_s(G) \leq (n-r) \binom{c-r}{s-1}+\binom{r}{s}=f_s(n,c-r,c) \leq \max \{f_s(n,k,c), f_s(n,t,c)\},
\]
as  the function $f_s(n,x,c)$ is convex for $x\in [k,t]$ and
$k\leq c-r\leq t$. This is a contradiction. This proves Claim
\ref{claim-convex}.
\end{proof}

\begin{claim}
$G$ contains a cycle of length at least $c$.
\end{claim}
The proof of the claim above is the same as Kopylov's proof and we skip it.
The proof of Theorem \ref{Thm_KWLcycle} is complete. {\hfill$\Box$}

Similar to Theorem \ref{Thm_KWLcycle}, we have the following result and
skip the details of the proof.
\begin{thm}
If  $G$ is  an $n$-vertex connected graph containing no $P_l$ and having   minimum degree $\delta(G)\geq k$, where $n\geq l\geq4$, then
$N_s(G)\leq \max\{f_s(n,k,l-1), f_s(n,\lfloor\frac{l}{2}\rfloor-1,l-1)\}$.
\end{thm}

\section{Consecutive lengths of cycles}
 For a graph $G$, let $\mu(G)$ be the largest eigenvalue of the adjacency matrix.
Nikiforov \cite{N08} proved the following:  If  $G$ is  a graph of sufficiently
large order $n$ and  the spectral radius $\mu(G)>\sqrt{\lfloor n^2/4\rfloor}$,
then $G$ contains a cycle of length $t$ for every $t\leq n/320$. We slightly improve Nikiforov's result as follows.

\begin{thm} \label{nikigenera}
Let $G$ be a graph of sufficiently large order $n$ with $\mu(G)>\sqrt{\lfloor n^2/4 \rfloor}$. Then $G$ contains a cycle of length $t$ for every $t \leq n/160$.
\end{thm}

Notice that Theorem \ref{Thm_Generkipas} implies the following fact:
\begin{fact}
A graph $G$ contains all cycles of length $t\in [3,l]$, where
$l=\frac{3N_3(G)}{N_2(G)}+2$.
\end{fact}

{\bf A sketch of the proof of Theorem \ref{nikigenera}.}
 Compared with
the original proof in \cite{N08}, the improvement comes
from the fact mentioned above. In \cite{N08}, it is shown that
for $n$ sufficiently large, there exists an induced subgraph
$H\subset G$ with $|H|>n/2$ satisfying one of the following conditions:
\begin{description}
\item[(i)] $\mu(H)>(1/2+1/80)|H|$;
\item[(ii)] $\mu(H)>|H|/2$ and $\delta(H)>2|H|/5$.
\end{description}
For case (i), it is shown in \cite{N08} that $N_3(H)\geq \frac{1}{960}|H|^3$.  In this case,  if
$e(H)=N_2(H)>\frac{|H|^2}{4}$, then a theorem of
Bollob\'as \cite{B} implies there are cycles of lengths
from $3$ to $\frac{|H|}{2}$ in $H$. Thus there are cycles
of length $t$ for each $3 \leq t \leq \tfrac{n}{4}$.
We  assume  $e(H)\leq \frac{|H|^2}{4}$. By Fact 1, $H$ contains all cycles of length $l\in [3,\frac{3|H|^3/960}{|H|^2/4}]$.
Since $\frac{3|H|^3/960}{|H|^2/4}\geq \frac{1}{160}n$,
we proved the result for the case (i). The proof  for case
(ii) follows from Nikiforov's  the original proof (see PP. 1497 in \cite{N08}).

\end{document}